\documentclass[10pt, 5p]{article} 

\usepackage{authblk}
\usepackage[utf8]{inputenc} 
\usepackage{amsmath}
\usepackage{amsfonts}
\usepackage{amssymb}
\usepackage{amsthm}
\usepackage{a4wide}
\usepackage{multicol}
\usepackage{booktabs}
\usepackage{url}
\usepackage{float}
\restylefloat{table}
\usepackage[square,numbers]{natbib}
\usepackage{graphicx} 
\usepackage[parfill]{parskip} 
\newtheorem{lemma}{Lemma}
\newtheorem{proposition}{Proposition}
\newtheorem{definition}{Definition}

\newtheorem{theorem}{Theorem}
\everymath{\displaystyle}
\newtheorem{example}{Example}

\usepackage{amssymb} 
\everymath{\displaystyle}
\usepackage{xcolor}

\usepackage{amsmath}
\usepackage{amsfonts}
\usepackage{amssymb}

\usepackage[authormarkup=footnote]{changes}

\definechangesauthor[color=magenta]{Carlos}
\definechangesauthor[color=red]{Sauleh}
\definechangesauthor[color=blue]{Mehdi}
\definechangesauthor[color=green]{Mathieu}

\date{}

\title{Convex Hull, IP and European Electricity Pricing in a European Power Exchanges setting with efficient computation of Convex Hull Prices}

\author[1]{Mehdi Madani\footnote{Corresponding author. Email: mmadani3@jhu.edu}}

\author[2]{Carlos Ruiz} 
\author[1,3,4]{Sauleh Siddiqui}
\author[5]{Mathieu Van Vyve}

\affil[1]{Department of Civil Engineering, Johns Hopkins University, United States}
\affil[2]{Department of Statistics, Universidad Carlos III de Madrid, Avda. de la Universidad 30, 28911 Leganés, Spain}
\affil[3]{Department of Applied Mathematics and Statistics, Johns Hopkins University, United States}
\affil[4]{German Institute for Economic Research (DIW Berlin), Germany}
\affil[5]{CORE - Voie du Roman Pays 34 bte L1.03.01, 1348 Louvain-la-Neuve, Belgium}

\begin{document}

%

\vspace{-0.2cm}
{\let\newpage\relax\maketitle}

\begin{abstract}
This paper introduces a computationally efficient comparative approach to  classical pricing rules for day-ahead electricity markets, namely Convex Hull Pricing, IP Pricing and European-like market rules, in a Power Exchange setting with non-convex demand bids. These demand bids can, for example, be useful to large industrial consumers, and extend demand block orders in use by European Power Exchanges. For this purpose, we show that Convex Hull Prices can  be efficiently computed using continuous relaxations for bidding products involving start-up costs, minimum power output levels \emph{and} ramp constraints, or analogous versions on the demand side. Relying on existing efficient algorithmic approaches to handle European-like market rules for such bidding products,  we provide comparative numerical experiments using realistic data, which, together with stylized examples, elucidates the relative merits of each pricing rule from  economic and computational perspectives. The motivation for this work is the prospective need for mid-term evolution of day-ahead markets in Europe and in the US, as well as the importance of day-ahead price signals, since these (spot) prices are used as reference prices for many power derivatives. The datasets, models and algorithms programmed in Julia/JuMP are provided in an online Git repository.



\end{abstract}



\newpage



%
%
%
%
%
%
%
%
%
%

\section{Introduction}
\label{paper4-intro}
\subsection{Day-ahead markets in the EU and the US}

European electricity markets, and day-ahead markets in particular, will evolve in the coming years as stakeholders have to cope with challenging issues such as the renewable energy revolution and a growing complexity due to an increasing number of market players.  

These real-world day-ahead electricity markets exhibit non-convexities in the underlying microeconomic/optimization models due to binary variables representing typical technical constraints such as minimum power output levels and special cost structures such as start-up costs.  Finding a market equilibrium \emph{supported by uniform prices} in such non-convex markets is known to be impossible under general conditions, see e.g. \cite{gribik2007,vanvyve,ruiz2012,mvv-revisiting} and the examples in Section \ref{subsec-examples} below. Uniform pricing means that in the market outcome, every market participant of a same market segment (location and hour
of the day) will pay or receive the same electricity price and no other transfers or payments are considered. To circumvent this difficulty, near-equilibrium prices are computed in practice, together with side payments where applicable. A given pricing rule specifies the defining properties of such near-equilibrium prices, and these pricing rules have differing computational complexity.

This article focuses on these microeconomic and computational issues and proposes a comparative approach to key theoretical pricing rules representative of main ISOs in the US and European Power Exchanges under the Price Coupling of Regions project. We first point out a few differences regarding how these markets are organized in the US and in Europe.

Market structures on both continents differ by the nature and role of the stakeholders. For example, day-ahead markets - which are the spot markets for electricity trading - are operated in the US by Independent Systems Operators (e.g., PJM, MISO, ERCOT) which are non-profit federally regulated organizations, while such markets are organized in the EU by Nominated Market Operators (NEMO) in the European legislation. The CACM guidelines released by the European Commission \cite{cacm} describe the legal framework in which these NEMOs (e.g., EPEX Spot, OMIE, Nord Pool) operate.

Also, it is well known that US markets allow market participants such as plant owners to describe their technical constraints and cost structure in a more granular way than bidding products used by European Power Exchanges. A particular example is the consideration of minimum up and down times of a unit which can not be directly described with current European bidding products. However, appropriately generalized (see the discussion in \cite{mvv-revisiting}), European bidding products can be used to describe stylized unit commitment problems where market participants can incur start up costs, minimum power output levels and ramp constraints. Compared to US markets, these European markets also consider potentially non-convex demand orders. In the present article, for comparison purposes, we will therefore consider bidding products including these key characteristics. They essentially provide stepwise bid curves for each hour of the day, a minimum acceptance ratio for the first step to describe a minimum power output level, and ramp constraints limiting the increase or decrease of production between each hour. They also associate a fixed cost to a given family of such bid curves, which is used to model a start up cost. 

\subsection{Contribution and structure of this article} \label{paper3-subsec:contrib}

The contribution here is threefold. First, this paper shows that in the presence of startup costs, ramp constraints and minimum power output levels, the convex hull of market participants' feasible sets  are given by their continuous relaxations. Then, using the so-called ``primal approach" \cite{vanvyve,schiro2015,hua2016}, Convex Hull Prices can efficiently be computed using continuous relaxations of the primal welfare maximization, described in Section \ref{section-chp}.  Second, the paper technically and succinctly describes Convex Hull Pricing and IP Pricing using unified notation in a context which also includes non-convex demand bids representing elastic demand, which is described in Section \ref{section-chp} and Section \ref{section-IPpricing} respectively. This is done in order to fit the ``European power exchange setting" where such kind of bids - though less general than those considered here - can be used by market participants. European-like market rules are then discussed in Section \ref{section-europeanrules} and references to previous works providing state-of-the-art models and algorithms are provided. We argue in particular that EU-like rules should be considered as a computationally-challenging variant of IP Pricing. Third, these contributions are used to present stylized examples illustrating each of the three approaches and highlighting advantages or potential drawbacks. Section \ref{section-numtests} is devoted to large-scale numerical tests comparing the three approaches from both the computational and economic perspectives. To the best of our knowledge, such comparative numerical tests on large-scale instances for each of these three key pricing rules, and in the presence of elastic demand, have not been proposed in the literature. Finally, the source code and datasets used for computational experiments are provided in an online Git repository \cite{dam_comp}.

\section{Notation, basic examples and marginal pricing} \label{section-notation-examples}

We consider a social welfare maximization program (SWP) described by (\ref{chp-primal-1})-(\ref{chp-primal-3}). For the sake of conciseness, only one location is considered, but all the developments which follow can straightforwardly be carried out when multiple locations are connected through a transmission network described by linear inequalities such as a DC approximation network model. Moreover, we consider a model with a multi-period structure in order to explicitly consider ramp constraints in the models below.

\subsection{Notation and MIP model}
The Social Welfare Maximization Program (SWP) is described by:

\begin{align}
&\max_{(u,x)} \sum_c B_c(u_c, x_c) \label{chp-primal-1}\\
&\mbox{s.t.}\notag\\
&\quad\sum_c \sum_{ic \in I_c | t(ic) = t} Q_{ic} x_{ic} = 0 & \forall t \in T &\ \  [\pi_t] \label{chp-primal-2}\\ 
&\quad x_c \in \mathbb{R}^I &\ \forall c \in C \\
&\quad u_c \in \{0,1\} &\ \forall c \in C \\
&\quad(u_c,x_c) \in X_c &\ \forall c \in C \label{chp-primal-3}
\end{align}

where $X_c$ is the set describing the technical constraints proper to participant $c\in C$, while $B(.)$ represents the costs of production (with $B<0$), or the utility of consumption (with  $B>0$) corresponding to production levels $Qx<0$ or consumption levels $Qx>0$. Here, $x_c \in \mathbb{R}^I$ is a vector whose components $x_{ic}$ correspond to the respective \emph{acceptance levels} of several bids $ic$ with bid quantities $Q_{ic}$, or several steps of a step-wise bid curve (assumed to be monotonic, increasing for offer orders and decreasing for demand orders), which will all be controlled by the binary variable $u_c$.

Regarding the sets $X_c$ and the cost or utility functions $B_c$, we will consider the special case corresponding to a stylized unit commitment setting where minimum output levels, start-up costs and ramp constraints are considered, but not for example minimum up and down times, which is more general than the bids considered in Europe. See \cite{euphemia} for a detailed description of the bidding products currently proposed by European power exchanges.

For stepwise bid curves with an associated start-up cost (or utility of demand reduced by a constant term), the utility or cost $B_c$ are such that (\ref{chp-primal-1}) is given by:

\begin{equation}\label{lin-primal-obj}
\max_{(u,x)} \sum_c (\sum_{ic\in I_c} P^{ic}Q^{}_{ic}x_{ic}  - F_c u_c) 
\end{equation}

The sets $X_c$ are described by the following binary requirements and linear inequalities (\ref{X-eq-binary})-(\ref{X-eq7}), where conditions (\ref{X-eq2}) describe minimum power output levels (the parameters $r_{ic} \in [0;1]$ are typically strictly positive only for the first step of an offer curve), while conditions (\ref{X-eq6}) impose upward limits on the output increase from one period to another and (\ref{X-eq7}) impose downward limits on output decreases (ramp constraints):

\begin{align}
&u_c \in \mathbb{Z} & \  & \  \label{X-eq-binary} \\
&x_{ic} \leq u_{c} & \forall ic \in I_c, c\in C \ & [s_{ic}^{max}] \label{X-eq1} \\
&x_{ic} \geq r_{ic }u_{c} & \forall ic \in I_c, c\in C \ & [s_{ic}^{min}] \label{X-eq2} \\
&u_{c} \leq 1 &\forall c \in C &[s_{c}] \label{X-eq3}\\
&u \geq 0 \label{X-eq4}
\end{align}

\begin{multline}
\sum_{ic \in I_c  | t(ic) = t+1 } (-Q^{ic})x_{ic} - \sum_{ic \in I_c | t(ic) = t} (-Q^{ic})x_{ic} \leq RU_c\ u_c \\ \forall t \in \{1,...,T-1\}, \forall c \in C \hspace{0.7cm}  [g^{up}_{c,t}] \label{X-eq6}
\end{multline}

\begin{multline}
  \sum_{ic \in I_c | t(ic) = t} (-Q^{ic})x_{ic} - \sum_{ic \in I_c | t(ic) = t+1} (-Q^{ic})x_{ic} \leq RD_c\ u_c \\ \forall t \in \{1,...,T-1\}, \forall c \in C \hspace{0.7cm}   [g^{down}_{c,t}] \label{X-eq7}
\end{multline}

\subsection{Examples} \label{subsec-examples}

We now present two basic examples which will be discussed throughout the article, illustrating key aspects both of markets with non-convexities and of the peculiar pricing rules considered in our contribution. We first use them to show the mathematical impossibility in general of a market equilibrium  \emph{supported by uniform prices} in the presence of non-convexities.

\bigskip
\begin{example} \label{example1} We consider a market with two buy bids (A and B) and two sell bids (C and D) where a minimum acceptance ratio, as described in Table \ref{table_toyexample_nc_1}, or a start up cost, as described in Table \ref{table_toyexample_nc_2}, are associated to the sell bid C. We will refer to them later on as Examples 1.1 and 1.2 respectively. Both types of non-convexities can obviously be combined.

\begin{figure}[ht]
\begin{center}
\includegraphics[scale=0.1]{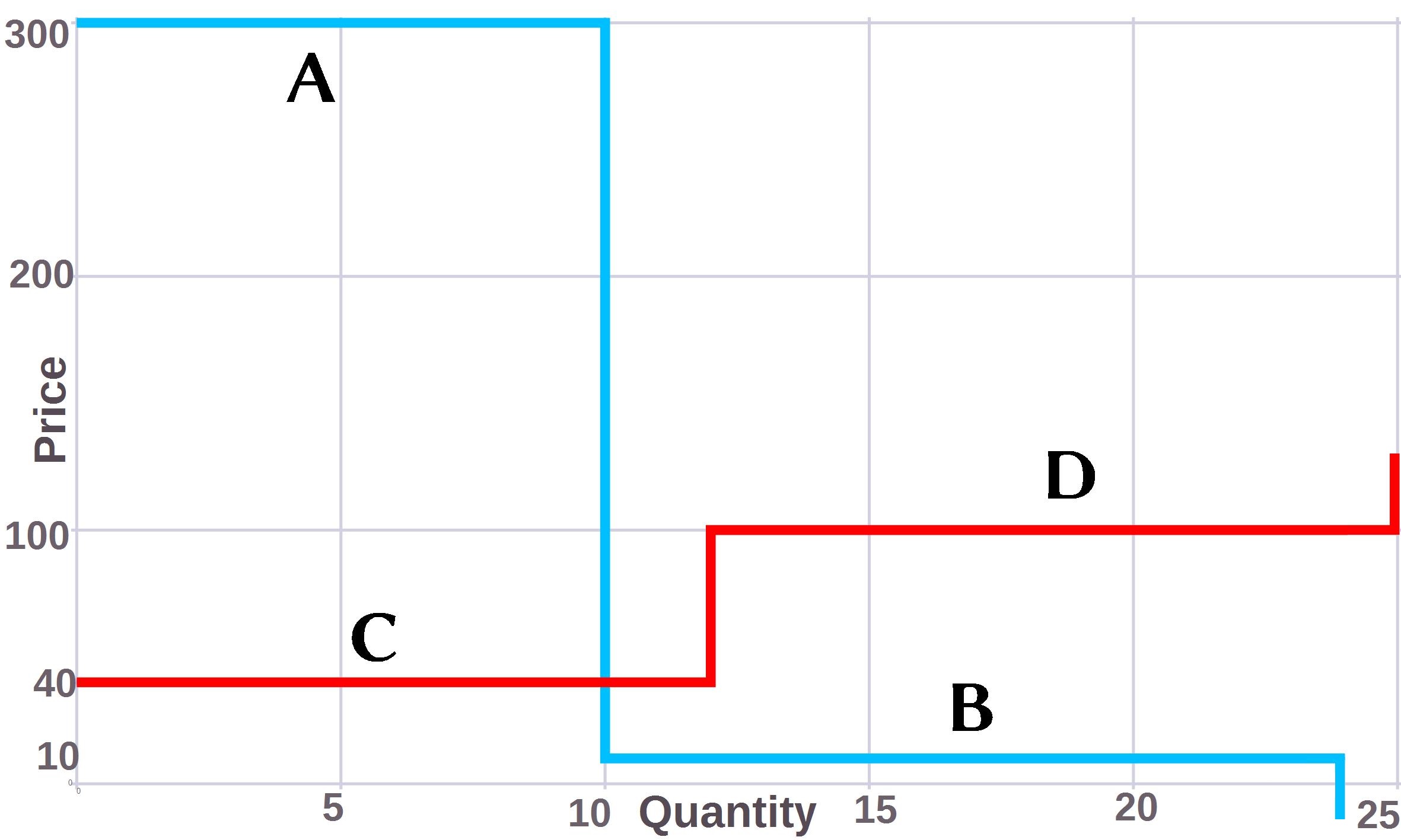}
\end{center}
\caption{Instance with a 'non-convex bid' C - start up cost or min. power output}
\end{figure}

\begin{table}[ht]
\begin{center}
\begin{tabular}{c|c|c|c}
Bids	&Quantity (MW)	&Limit price (EUR/MW)	&Min. Acceptance Ratio\\
\hline
A - Buy bid 	&10	 &300	& - \\
B - Buy bid  	&14  &10	& - \\
C - Sell bid 1	        &12	 &40	& $\frac{11}{12}$ \\
D - Sell bid 2	        &13  &100	& - \\
\end{tabular} 
\end{center}
\caption{Instance with a minimum acceptance ratio (minimum power output level)}
\label{table_toyexample_nc_1}
\end{table}

\begin{table}[ht]
\begin{center}
\begin{tabular}{c|c|c|c}
Bids	&Quantity (MW)	&Limit price (EUR/MW)	&Start-up costs\\
\hline
A - Buy bid (step 1)	&10	 &300	& - \\
B - Buy bid (step 2)	&14  &10	& - \\
C - Sell bid 1	        &12	 &40	& 200 \\
D - Sell bid 2	        &13  &100	& - \\
\end{tabular} 
\end{center}
\caption{Instance with start-up costs}
\label{table_toyexample_nc_2}
\end{table}
\end{example}

For example, the two toy examples presented above can readily be described as an instance of (\ref{chp-primal-1})-(\ref{chp-primal-3}) in the more special case described by (\ref{lin-primal-obj})-(\ref{X-eq7}): the left column below corresponds to the instance of Table \ref{table_toyexample_nc_1} and the right column to the instance of Table \ref{table_toyexample_nc_2}. Here, we drop the index $i$ as all the sets $I_c$ involved are singletons.

\vspace{2.5cm}

\begin{multicols}{2}
Example 1.1:
\begin{center}

$\max_{x,u}\ (10)(300)x_a + (14)(10)x_b - (12)(40)x_c - (13)(100)x_d$
\begin{align}
&\mbox{s.t.}\notag\\
&\quad 10x_a + 14x_b - 12x_c - 13x_d = 0 \\
&\quad x_a \leq u_a \\
&\quad x_b \leq u_b \\
&\quad x_c \leq u_c \\
&\quad x_c \geq (11/12)u_c \\
&\quad x_d \leq u_d \\
&\quad x \geq 0 \\
&\quad u \in \{0,1\}^4
\end{align}
\end{center}

\columnbreak
Example 1.2:
\begin{center}
$\max_{x,u}\ (10)(300)x_a + (14)(10)x_b - (12)(40)x_c - (13)(100)x_d - 200 u_c$
\begin{align}
&\mbox{s.t.}\notag\\
&\quad 10x_a + 14x_b - 12x_c - 13x_d = 0 \\
&\quad x_a \leq u_a \\
&\quad x_b \leq u_b \\
&\quad x_c \leq u_c \\
&\quad x_d \leq u_d \\
&\quad x \geq 0 \\
&\quad u \in \{0,1\}^4
\end{align}
\end{center}

\end{multicols}

The binary variables $u_a, u_b$ and $u_d$ can readily be set to 1 and removed from the formulation: they are associated to the simple ``convex bids" A, B and D, and actually not required as it is always optimal to set them to 1.

In the first case of a minimum acceptance ratio, the pure welfare maximizing solution is to accept C at its minimum acceptance level of $(11/12)$, that is to accept 11 MW from C, to fully accept A, and to accept the fraction of B needed to match the accepted fraction of C. For a market equilibrium to exist, the market price should be 10 EUR/MW, set by B which is fractionally accepted: otherwise, there would be either some leftover demand from B if the price is below, or B would prefer to be fully rejected if the price is above. However, at this market price, C is loosing $11(40-10)=330$ EUR and would therefore prefer not to be dispatched. Hence, there is no market equilibrium with uniform prices in the present case.

In the second case of the presence of start up costs, it can be easily checked that the pure welfare maximizing solution is to fully accept A, fully reject B, and accept the fraction of C needed to match A. Any level of acceptance of B would inevitably degrade welfare as the bid price of B is lower than the bid price of any other offer bid, and also, it can be readily checked that discarding C in order to avoid the associated start-up cost would also lead to less welfare (see the discussion of European rules below). The optimal welfare is hence ``the utility of A minus costs of the production by C", that is $10(300) - [10(40) + 200] = 2400$. Here again, if there is any market equilibrium supported by uniform prices, the price is set by fractionally accepted bids, here by C at 40 EUR/MW. However, at such a market price, C doesn't recover its start up costs and would prefer to be rejected: there is no market equilibrium with uniform prices.

\medskip

Let us note that in general, requiring equilibrium for the convex part of the problem, that is for the ``convex bids" which do not include start up costs or indivisibilities, as is the case for EU market rules and IP Pricing may be questionable. This is further discussed in their respective sections below. The following example aims at illustrating this aspect.

\medskip

\begin{example}\label{example2}
This example is described in Table \ref{table_toyexample_ippricing_nonintuitive}
\end{example}

\begin{table}[ht]
\begin{center}
\begin{tabular}{l|c|c|c}
Bids	&Quantity (MW)	&Limit price (EUR/MW)	&Min. Acc. Ratio\\
\hline
A - Sell bid 	        &50	 &30	    & - \\
B - Buy bid 	        &50  &130     	& - \\
C - Sell bid	        &40	 &40	    & - \\
D - Sell block bid      &200  &60	    & 1 \\
E - Buy block bid       &200  &90	    & 1 \\
\end{tabular} 
\end{center}
\caption{Instance with non-intuitive 'IP pricing' outcome}
\label{table_toyexample_ippricing_nonintuitive}
\end{table}

Concerning input data in this Example \ref{example2}, let us recall that fully indivisible bids (so-called block bids in EU markets) could correspond to real technical conditions of power plants as reported in \cite{ring}, p.9 concerning ``combustion turbine units for which the minimum and maximum outputs are the same."

It could be straightforwardly shown, for example by solving the corresponding MILP problem, that the welfare maximizing solution is here given by fully accepting A, B, D, E and rejecting C. 

If one imposes the constraints that the convex bids should be at equilibrium, as C is fully rejected, the commodity market price must be less than or equal to 40 EUR/MW (the marginal cost of C), and as A is fully accepted, the market price must be greater than or equal to 30 EUR/MW. For such a market price in [30;40], D is ``paradoxically accepted" with respect to the commodity price. In the case of IP Pricing discussed below in Section \ref{section-IPpricing}, it would receive a compensating ``start up price" as a make-whole payment. However, intuitively, one may prefer to set the price for example at 75 EUR/MW, in between the ``marginal costs" and ``marginal utility" of D and E respectively, in which case no make-whole payment is needed. If the bid C which is not part of the welfare maximizing solution is removed from the instance, such an outcome would correspond to a market equilibrium based on the commodity price only.

\section{Convex hull pricing with ramp constraints} \label{section-chp}

We now show how Convex Hull Prices (CHP)  can be efficiently computed for the stylized offer or demand bids presented above, including ramp constraints conditions. We also review their general uplift minimizing property in that specific setting. Convex Hull Pricing was first proposed in \cite{gribik2007}. Ring has proposed \cite{ring} to minimize so-called uplifts -- a formal definition is provided below -- made to market participants to compensate them from the actual losses or opportunity costs they face at the computed market prices. The key contribution in \cite{gribik2007} has been to show how to compute market prices minimizing the corresponding required uplifts using Lagrangian duality (see \cite{geoffrion1974} on classical Lagrangian duality results). Prices obtained are sometimes also called Extended Locational Marginal Prices (ELMP) \cite{wang2013,schiro2015}. The approach is of current interest in the US, see for example the contribution \cite{schiro2015} by researchers at the ISO New England.

In what follows, for notational convenience, when a multi-period setting is considered, we still use the compact notation applicable to a single-period setting, that is, formally:

\begin{center}
$\pi\sum_{ic \in I_c} Q_{ic} x_{ic} := \sum_{t}\pi_t \sum_{ic \in I_c | t(ic) = t} Q_{ic} x_{ic}$
\end{center}

Given an optimal solution $(u^*,x^*)$ and a market price $\pi$, the uplift of participant $c \in C$ is defined as:

\begin{center}
\begin{equation}
uplift_{(u_c^*,x_c^*)}(\pi) := \left(
\max_{(u_c, x_c) \in X_c} \left[ B_c(u_c, x_c) - \pi\sum_{ic \in I_c} Q_{ic} x_{ic} \right] \right) - \left( B_c(u^*_c, x^*_c) -  \pi\sum_{ic \in I_c} Q_{ic} x^*_{ic} \right)
\end{equation}
\end{center}

The interpretation is straightforward: the uplift is the gap between the maximum surplus participant $c$ could extract facing the market price $\pi$ by choosing the best option regarding only its own technical constraints, and the surplus obtained with this same market price and the welfare maximizing solution. This gap is trivially always non-negative.

The contribution \cite{gribik2007} has shown that market price(s) such that the sum of all these uplifts is minimal can be obtained by solving the Lagrangian dual of the welfare maximizing program (\ref{chp-primal-1})-(\ref{chp-primal-3}) where only the balance constraint(s) (\ref{chp-primal-2}) have been dualized. Indeed, \cite{gribik2007} considers a context where costs of production to serve a given load $y$ should be minimized, but can  be adapted to our context of two-sided auctions with both offers and demands.  We review here this result, specializing the model in \cite{gribik2007} to the present context and notation.

\bigskip

\begin{theorem}
Let $\pi^*$ solve the Lagrangian dual of (\ref{chp-primal-1})-(\ref{chp-primal-3}) where the balance constraint(s) (\ref{chp-primal-2}) have been dualized:

\begin{center}
\begin{equation}
\min_{\pi} \left[\max_{(u_c,x_c) \in X_c, c \in C} \left[ \sum_c B_c(u_c, x_c) - \pi\sum_c \sum_{ic \in I_c} Q_{ic} x_{ic} \right] \right] \label{chp-lagrange-1}
\end{equation}
\end{center}

Then, $\pi^*$ solves:

\begin{center}
\begin{equation}
\min_{\pi} \sum_c uplift_{(u_c^*,x_c^*)}(\pi)
\end{equation}
\end{center}

\end{theorem}

\begin{proof}
As the lower level program is separable in $c \in C$, the dual (\ref{chp-lagrange-1}) can equivalently be written as:

\begin{center}
\begin{equation}
z^* = \min_{\pi} \left[\sum_c \max_{(u_c,x_c) \in X_c, c \in C}  \left[ B_c(u_c, x_c) - \pi \sum_{ic \in I_c} Q_{ic} x_{ic} \right] \right] \label{chp-lagrange-1b}
\end{equation}
\end{center}

Let us observe that under constraint(s) (\ref{chp-primal-2}), we have:

\begin{center}
\begin{equation}
\sum_c B_c(u_c, x_c) = \sum_c B(u_c, x_c) - \pi\sum_c \sum_{ic} Q_{ic} x_{ic}
\end{equation}
\end{center}

Hence (\ref{chp-primal-1})-(\ref{chp-primal-3}) can equivalently be written with an \emph{arbitrary} $\pi$ as:

\begin{center}
\begin{equation}\label{chp-lagrange-2}
w^*(\pi) = w^* = \max  \left[ \sum_c \left[ B_c(u_c, x_c)- \pi\sum_{ic \in I_c} Q_{ic} x_{ic} \right] \right] 
\end{equation}
\end{center}

\begin{align}
&\sum_c \sum_{ic} Q_{ic} x_{ic} = 0 \label{chp-lagrange-2-constr1} \\  
&(u_c,x_c) \in X_c &\ \forall c \in C   \label{chp-lagrange-2-constr2}
\end{align}

By weak duality,  $w^* \leq z^*$. Moreover, as now detailed, the duality gap $DG = z^* - w^*$ exactly corresponds to the sum of the uplifts, and solving the Lagrangian dual hence aims at minimizing these. Again, let $(u^*, x^*)$ be a welfare optimal solution, i.e.  solving (\ref{chp-lagrange-2})-(\ref{chp-lagrange-2-constr2}), then $DG = z^* - w^*$ can be written as:

\begin{center}
\begin{equation}
\min_{\pi} \left[\sum_c \max_{(u_c,x_c) \in X_c, c \in C}  \left[ B_c(u_c, x_c) - \pi \sum_{ic \in I_c} Q_{ic} x_{ic} \right] - \sum_c \left[ B_c(u_c^*, x_c^*)- \pi\sum_{ic \in I_c} Q_{ic} x^*_{ic} \right] \right]
\end{equation}
\end{center}

or equivalently as:

\begin{center}
\begin{equation}
\min_{\pi} \left[\sum_c \left( \max_{(u_c,x_c) \in X_c, c \in C}  \left[ B_c(u_c, x_c) - \pi \sum_{ic \in I_c} Q_{ic} x_{ic} \right] -  \left( B_c(u_c^*, x_c^*)- \pi\sum_{ic \in I_c} Q_{ic} x^*_{ic} \right) \right)\right]
\end{equation}
\end{center}

This shows that solving the Lagrangian dual with the balance constraints dualized provides prices minimizing the sum of the uplifts.\end{proof}

As first observed in \cite{vanvyve} and more recently in \cite{schiro2015} and \cite{hua2016}, solving the Lagrangian dual can in certain situations be reduced to solving the continuous relaxation of the primal (\ref{chp-primal-1})-(\ref{chp-primal-3}). This holds when this continuous relaxation is itself equivalent to the following ``equivalent" formulation (under rather mild assumptions requiring the $X_c$ to be compact mixed integer linear sets, see \cite{hua2016}) of the Lagrangian dual to consider:


\begin{align}
&\max \sum_c B^{**}_{c, X_c}(u_c, x_c) \label{chp-primal-conv-obj}\\
&\mbox{s.t.}\notag\\ 
&\quad \sum_c \sum_{ic} Q_{ic} x_{ic} = 0 & [\pi] \label{chp-primal-conv-2}\\ 
&\quad (u_c,x_c) \in conv(X_c) &\ \forall c \in C \label{chp-primal-conv-3}
\end{align}

where $conv(X_c)$ denotes the convex hull of the feasible set $X_c$, and $B^{**}_{c, X_c}$ the convex envelope of $B_c$ taken over $X_c$, i.e., the lowest concave over-estimator  of $B_c$ on $conv(X_c)$, see \cite[Theorem 1]{hua2016} in a different setting where costs are minimized instead of welfare maximized. See also the underlying results in \cite{falk1969} used therein, or also \cite{geoffrion1974} for equivalent results in a mixer integer linear setting. In such a case, the optimal dual variables $\pi^*$ related to the constraint(s) (\ref{chp-primal-2}) of the continuous relaxation, which can often be obtained as a by-product when solving this continuous relaxation, provide an optimal solution to the Lagrangian dual (\ref{chp-lagrange-1}).

Following \cite{hua2016}, if $B_c$ are linear functions (the \emph{marginal} costs/utilities are constant), $B_c$ and $B^{**}_{c, X_c}$ have the same ``functional forms" and we are only required to describe $conv(X_c)$ appropriately. This motivates the study of possible polyhedral representations for $X_c$, for instance, a review is given in \cite{hua2016} which also considers quadratic cost functions and their convex envelopes over the $X_c$.

In our context, as the $B_c$ are linear functions, all we need is a description of $conv(X_c)$ where the $X_c$ are given above by (\ref{X-eq-binary})-(\ref{X-eq7}), which is stated in Theorem \ref{thm-chp-cr} below. The following Lemma is first required:

 \bigskip
 
\begin{lemma}\label{lemma:cvx-Axb}
Consider a polyhedron $P$ in $\mathbb{R}^n$ described by conditions $Ax \leq b$ and the set in $\mathbb{R} \times \mathbb{R}^n$  $X = \{(0,0)\} \cup \{(1,x) | x \in P \}$. Then $conv(X) = \{(u,x) \in \mathbb{R}\times \mathbb{R}^n | \ 0\leq u \leq 1, Ax \leq b u\}$.
\end{lemma}

\begin{proof}
As $P$ is convex, the only case to consider is a convex combination of $(0,0)$ and $(1,x)$ where $x$ satisfies $Ax \leq b$ (the other cases are trivial).

Any such convex combination can be written as $(u, ux)$ for some $0 \leq u \leq 1$. Moreover, for any $0 < u \leq 1$,  $Ax \leq b \Leftrightarrow A ux \leq ub \Leftrightarrow A\tilde{x} \leq bu$ with $\tilde{x} = ux$. \\ Hence $conv(X) = \{ (u, \tilde{x}) | A\tilde{x} \leq bu, 0 \leq u \leq 1  \}$ which proves the result.
\end{proof}

We come to a key Theorem which shows that Convex Hull Prices can be computed efficiently in the presence of start-up costs, minimum power output levels, and ramp constraints. This follows from the fact that the convex hull of each market participant's feasible set is in this setting  obtained by taking the continuous relaxation in the space of the original variables and doesn't require to consider an extended space, i.e. adding auxiliary variables, to describe a set of which $conv(X_c)$ would be the projection. Note that we need to assume that minimum up and down times are not considered, which is the case for the bidding products we consider here and which are slightly more general than those proposed by European power exchanges. Describing the convex hull of market participant's feasible sets in the presence of minimum up and down times requires to consider extended spaces and currently known formulations would be less efficient for large-scale instances; see the review in \cite{hua2016}.

\bigskip

\begin{theorem} \label{thm-chp-cr}
Consider the market participant's feasible set $X_c$ described by (\ref{X-eq-binary})-(\ref{X-eq7}). Then $conv(X_c)$ is described by the continuous relaxation of $X_c$, i.e. by (\ref{X-eq1})-(\ref{X-eq7}).
\end{theorem}

\begin{proof}
As $u_c=0 \Rightarrow x_c=0$, obviously, $X_c = (0,0) \cup \{(1,x) | Ax \leq b\}$ where the conditions $Ax \leq b$ correspond to conditions (\ref{X-eq-binary})-(\ref{X-eq7}) written with $u_c =1$ (appropriately choosing $A,b$).

The result is hence a direct consequence of Lemma \ref{lemma:cvx-Axb}.
\end{proof}

Let us now observe the outcome Convex Hull Pricing (CHP) gives on the Examples described above. In the context of Examples 1.1 and 1.2, the sets $X_c$ are described by $r_c u_c \leq x_c \leq u_c, u_c \in \{0,1\}$, where $r_c$ is respectively $(11/12)$ and $0$. It is trivial to verify that in these cases, $conv(X_c)$ is described by its continuous relaxation, i.e. by $r_c u_c \leq x_c \leq u_c, 0\leq u_c \leq 1$ and this is also a simple special case of Theorem \ref{thm-chp-cr}.

\textbf{Example 1 (continued): Convex Hull Pricing case}

\begin{multicols}{2}
Example 1.1 (CHP case):

\begin{center}

$\max_{x,u}\ (10)(300)x_a + (14)(10)x_b - (12)(40)x_c - (13)(100)x_d$
\begin{align*}
&\mbox{s.t.}\notag\\
&\quad 10x_a + 14x_b - 12x_c - 13x_d = 0 & [\pi^* = 40]\\
&\quad x_a \leq 1 \\
&\quad x_b \leq 1 \\
&\quad x_c \leq u_c \\
&\quad x_c \geq (11/12)u_c \\
&\quad x_d \leq 1 \\
&\quad x\geq 0 \\
&\quad 0\leq u_c \leq 1
\end{align*}
\end{center}

\columnbreak

Example 1.2 (CHP case):

\begin{center}
$\max_{x,u}\ (10)(300)x_a + (14)(10)x_b - (12)(40)x_c - (13)(100)x_d - 200 u_c$
\begin{align*}
&\mbox{s.t.}\notag\\
&\quad 10x_a + 14x_b - 12x_c - 13x_d = 0 & [\pi^* = 56.6...] \\
&\quad x_a \leq 1 \\
&\quad x_b \leq 1 \\
&\quad x_c \leq u_c \\
&\quad x_d \leq 1 \\
&\quad x \geq 0 \\
&\quad 0\leq u_c \leq 1
\end{align*}
\end{center}

\end{multicols}

Hence, the outcomes are:

\begin{multicols}{2}

Example 1.1 (CHP case):

\begin{enumerate}
\item Welfare maximizing solution: fully accept A, accept $(11/12)$ of C, accept $(1/14)$ of B, fully reject D.

\item Market price: $\pi = 40$

\item Uplifts: no uplift for A, C, D, while $B$ requires an uplift of 30 EUR.

\end{enumerate}

\columnbreak

Example 1.2 (CHP case):

\begin{enumerate}
\item Welfare maximizing solution: fully accept A, accept $(10/12)$ of C, fully reject B, D.

\item Market price: $\pi = 56.6...$

\item Uplifts: no uplift for A, B, D, while $C$ requires an uplift of

 $[(12)56.6.. - ((12)40 + 200)] \\ - [(10)56.6.. - ((10)40 + 200)] \\ =  0 - (-33.333..) = 33.333..$
\end{enumerate}

\end{multicols}

One can readily check by solving the corresponding LP that the sum of the uplifts, respectively of 30 EUR and 33.333 EUR, correspond to the duality gaps.

\textbf{Example 2 (continued) : Convex Hull Pricing case}

Finally, let us consider Example \ref{example2} presented in Section \ref{section-notation-examples}. Solving the continuous relaxation, i.e. leaving aside that D and E can only be fully accepted or fully rejected, the optimal dual variable value of the balance constraint gives an uplift minimizing price of 60 EUR/MW. Only C requires an uplift, as at that price the participant would prefer to have the bid fully accepted, with a surplus of $40(60-40)=800$ instead of 0, the uplift hence being of 800 EUR. Let us note that one can consider market rules where only uplifts for actual losses and not those corresponding to opportunity costs would be paid to market participants. Though the outcome may be here more intuitive than when IP pricing is used (cf. the analogue example discussed in Section \ref{section-IPpricing}), the price is still influenced by bids rejected in the welfare maximizing solution, as the uplifts could correspond either to ``opportunity costs" or to ``actual losses" incurred.

\section{IP Pricing} \label{section-IPpricing}

This Section presents IP Pricing in the power exchange setting with non-convex demand bids and related notation considered in this paper. Compared to the historical exposition in \cite{oneill}, we also explicitly consider the economic interpretations of the so-called ``commitment prices" as potential losses or opportunity costs of accepted or rejected non-convex bids respectively (see Propositions \ref{proposition-da} and \ref{proposition-dr} below). These propositions are used to prove the equilibrium property of the IP Prices given by Theorem \ref{thm-ip-equilibrium}.

IP Pricing was first introduced by O'Neill et al. in \cite{oneill}. The proposition is to determine prices by using the convex part of the welfare maximization problem: \emph{roughly speaking} ``marginal units in the chosen unit commitment and dispatch are setting the price." More precisely, the approach proposed is to (a) maximize welfare, (b) fix all binary variables to the optimal values found, (c) derive commodity (electricity) prices as optimal dual variables of the balance constraints - as usual to determine locational marginal prices - and (d) start up prices (or commitment prices) as optimal dual variables to the constraints fixing the binary variables to their optimal value. The key contribution is to show that the derived price system supports a market equilibrium \emph{if the market rules specify that payments appropriately depend on both kinds of prices} (Theorem 2 of the original paper). This equilibrium property is derived in Theorem \ref{thm-ip-equilibrium} with the appropriate settlement rule formally given in Definition \ref{def-settlement} below.

The fact that the commodity prices are derived as optimal dual variables of the balance constraints in the restricted welfare maximizing problem where integer decisions are fixed implies that ``marginal units" (here whose production or consumption level is partial with regard to their technical capabilities) are setting the price. This could be derived from the equilibrium properties of prices obtained in well-behaved convex contexts such as for the welfare maximization programs obtained once the binary decisions have been fixed.

We now formally derive these results and then illustrate them with our key toy examples, see the continuation of Example \ref{example1} and Example \ref{example2} discussed below.

Let us consider the following fixing constraints given a partition of $C$ in accepted bids $C_a$ and rejected bids $C_r$. In the original IP Pricing proposition, this partition is the one given by the optimal solution to the primal welfare maximizing program (SWP):

\begin{align}
&u_{c_a} = 1 & \forall c_a \in C_a \subseteq C & \ \ \ [\delta_{c_a}] \label{fix1cond} \\
&u_{c_r} = 0 & \forall c_r \in C_r \subseteq C & \ \ \ [\delta_{c_r}] \label{fix0cond}
\end{align}

The dual variables $g^{up}_{c,t}$ and $g^{down}_{c,t}$ associated to (\ref{X-eq6})-(\ref{X-eq7}) do not exist for $t=T$ or $t=0$. However, for technical convenience to develop what follows, we set $g^{up}_{c,0} = g^{up}_{c,T}  = g^{down}_{c,0} = g^{down}_{c,T} = 0$.

The optimization problem dual to the welfare maximizing program where the $u_c$ have fixed values (SWP-FIXED) described by (\ref{lin-primal-obj}) and conditions, (\ref{chp-primal-2}), (\ref{X-eq1})-(\ref{X-eq2}), (\ref{X-eq6})-(\ref{X-eq7}) and (\ref{fix1cond})-(\ref{fix0cond}), is given by:
%
%
\begin{align}
&\min_{s, \pi, \delta} \sum_{c_a \in C_a \subseteq C} \delta_{c_a} \label{swp-fixed-dual-obj}\\
&\mbox{s.t.}\notag\\
&\quad s_{ic}^{max} - s_{ic}^{min} + (Q^{ic}g^{down}_{c,t(ic)-1} - Q^{ic}g^{up}_{c,t(ic)-1}) \notag\\
&\qquad \qquad \qquad \qquad \quad+ (Q^{ic}g^{up}_{c,t(ic)} - Q^{ic}g^{down}_{c,t(ic)}) + Q_{ic} \pi_{t(ic)} = P^{ic}Q_{ic} & [x_{ic}]  \label{cc-dual-xic} \\
&\quad\delta_{c_a} \geq \sum_{ic_a \in I_{c_a}} (s_{ic_a}^{max} - r_{ic_a}s_{ic_a}^{min}) - F_{c_a} + \sum_t (RU_{c_a} g^{up}_{c_a,t} + RD_{c_a} g^{down}_{c_a,t}) & [u_{c_a} := 1] \label{cc-dual-u1} \\
&\quad\delta_{c_r} \geq \sum_{ic_r \in I_{c_r}} (s_{ic_r}^{max} - r_{ic_r}s_{ic_r}^{min}) - F_{c_r} + \sum_t (RU_{c_r} g^{up}_{c_r,t} + RD_{c_r} g^{down}_{c_r,t}) & [u_{c_r} := 0] \label{da-cond} \\
&\quad s^{max}, s^{min}\geq 0
\end{align}

\begin{lemma}\label{lemma-profit-loss}
The following identity holds for an accepted bid $c$ (hence $u_c = 1$) and decomposes the profit or loss of the market participant $c$ in terms of 
particular economic surplus variables: 
\begin{center}
$\sum_{ic} -Q_{ic}(\pi_{t(ic)}-P^{ic})x_{ic} = \sum_{ic \in I_{c}} (s_{ic}^{max} - r_{ic}s_{ic}^{min})  + \sum_t (RU_c g^{up}_{c,t} + RD_c g^{down}_{c,t})$ 
\end{center}
\end{lemma}

\begin{proof}
See appendix. \end{proof}

\begin{proposition} \label{proposition-da}
For an accepted bid $c_a$, the dual variable  $\delta_{c_a}$ corresponds to the  profit if positive, or the loss if negative, of the market participant $c$, which is given by $\sum_{ic_a} -Q_{ic_a}(\pi_{t(ic_a)}-P^{ic_a})x_{ic_a} - F_{c_a}$
\end{proposition}

\begin{proof}
The complementarity conditions associated to constraints (\ref{cc-dual-u1}) satisfied by any pair of respectively primal and dual optimal solutions are given by:
\begin{center}
$u_{c_a}(\delta_{c_a} - \sum_{ic_a \in I_{c_a}} (s_{ic_a}^{max} - r_{ic_a}s_{ic_a}^{min}) + F_{c_a} - \sum_t (RU_{c_a} g^{up}_{c_a,t} + RD_{c_a} g^{down}_{c_a,t})) = 0$ with $u_{c_a}:=1$
\end{center}

Hence, $\delta_{c_a} = \sum_{ic_a \in I_{c_a}} (s_{ic_a}^{max} - r_{ic_a}s_{ic_a}^{min})  + \sum_t (RU_{c_a} g^{up}_{c_a,t} + RD_{c_a} g^{down}_{c_a,t}) - F_{c_a}$ and the results immediately follows from Lemma \ref{lemma-profit-loss}.
\end{proof}

\begin{proposition} \label{proposition-dr}
The dual variables $\delta_{c_r}$, when constrained by (\ref{da-cond}) to be positive, are upper bounds on the missed profit of market participant $c$ facing prices $\pi_t$. Note that in all cases, it is straightforward to see that in the dual, one can always consider the optimal solutions such that (\ref{da-cond}) is tight, as $\delta_{c_r}$ doesn't appear in the dual objective nor elsewhere in the dual.
\end{proposition}

\begin{proof}
Let $x^*_{ic}$ correspond to optimal decisions for $c$ if it was committed under its own technical constraints when facing given prices $\pi_t$. In what follows, it is important to keep in mind that the dual variables, roughly speaking the prices $\pi_t$ and related quantities on the ``price side", together with the relations they satisfy, are given and independent of what would be ``on the quantity side" the optimal choices $x^*$ of $c$ optimizing when facing them.

In case of commitment, $u_c:=1$, and multiplying (\ref{X-eq6})-(\ref{X-eq7}) respectively by the dual variables $g^{up}_{c,t}$ and $g^{down}_{c,t}$, setting $u_c:=1$ and summing the equations obtained gives:

\begin{equation}\label{majoration-1}
\sum_t (RU_c g^{up}_{c,t} + RD_c g^{down}_{c,t}) \geq \sum_{ic} \left( (Q^{ic}g^{down}_{c,t(ic)-1} - Q^{ic}g^{up}_{c,t(ic)-1})x^*_{ic} + (Q^{ic}g^{up}_{c,t(ic)} - Q^{ic}g^{down}_{c,t(ic)})x^*_{ic} \right)
\end{equation}

On the other hand, as $s^{max},s^{min} \geq 0$ and the parameters $r_{ic} \in [0;1]$, one has:

\begin{equation}\label{majoration-2}
(s_{ic_r}^{max} - r_{ic_r}s_{ic_r}^{min}) \geq (s_{ic_r}^{max} - s_{ic_r}^{min})
\end{equation}

Multiplying (\ref{cc-dual-xic}) by $x^*_{ic}$, summing up over the $ic$ and using the upper bounds given by (\ref{majoration-1}) and (\ref{majoration-2}) shows after rearrangements that the right-hand side of (\ref{da-cond}) is bounded from below by the missed profit:

\begin{equation}
\sum_{ic_r \in I_{c_r}} (s_{ic_r}^{max} - r_{ic_r}s_{ic_r}^{min}) - F_{c_r} + \sum_t (RU_c g^{up}_{c,t} + RD_c g^{down}_{c,t}) \geq \sum_{ic} -Q_{ic}(\pi_{t(ic)}-P^{ic})x^*_{ic} - F_c
\end{equation}

The result then follows from (\ref{da-cond}) \end{proof}

The following Theorem states the equilibrium properties of the price system defined by $\pi_t$ (commodity prices) and the $\delta_c$ (commitment prices). \emph{It should be noted that Propositions \ref{proposition-da} \& \ref{proposition-dr} and Theorem \ref{thm-ip-equilibrium} hold whatever is the bid selection given by (\ref{fix1cond})-(\ref{fix0cond}): nothing in the proofs relies on the fact that the binary values in the fixing constraints (\ref{fix1cond})-(\ref{fix0cond}) are those obtained as optimal values of the welfare maximizing program.}
\medskip

As mentioned in the introduction to this Section, Theorem \ref{thm-ip-equilibrium} shows that the non-uniform price system $(\pi^*, \delta^*)$ supports a market equilibrium \emph{provided the settlement rule is appropriately defined as in Definition \ref{def-settlement} below}. It yields a market equilibrium in the sense that, given the market rules and the price system, the market participants could not be better off by choosing other production/consumption decisions satisfying their technical constraints.

\medskip

\begin{definition}[\textbf{IP Pricing Settlement Rule}] \label{def-settlement}
\ 
Given the market prices $(\pi^*, \delta^*)$:
\begin{itemize}
\item each seller $c$ \emph{is paid} $[ \sum_{ic\in I_c}  \pi^*_{t(ic)} (-Q^{}_{ic}x_{ic}) -\delta^*_c u_c ]$
\item each buyer $c$ \emph{pays} $-[ \sum_{ic\in I_c}  \pi^*_{t(ic)} (-Q^{}_{ic}x_{ic}) -\delta^*_c u_c ]$ 
\end{itemize} 
\end{definition}

\begin{theorem}[analogue of Theorem 2 in \cite{oneill}]\label{thm-ip-equilibrium}
For the price system given by $\pi^*, \delta^*$ obtained as dual variables to the constraints (\ref{chp-primal-2}) and (\ref{fix1cond})-(\ref{fix0cond}) respectively, the primal decision variable values $(x_c, u_c)$ obtained in (SWP-FIXED) (given by (\ref{lin-primal-obj}) and conditions, (\ref{chp-primal-2}), (\ref{X-eq1})-(\ref{X-eq2}), (\ref{X-eq6})-(\ref{X-eq7}) and (\ref{fix1cond})-(\ref{fix0cond})) are solving (the first bracketed term correspond to payments and the second to costs or utility):

\begin{equation}
\max_{u_c, x_c} [ \sum_{ic\in I_c}  \pi^*_{t(ic)} (-Q^{}_{ic}x_{ic}) -\delta^*_c u_c ]  - [\sum_{ic\in I_c} P^{ic} (-Q^{}_{ic}x_{ic})  +  F_c u_c  ]      \label{indiv-obj}
\end{equation}
s.t.
\begin{equation}\label{indiv-conds}
(u_c, x_c) \in X_c,
\end{equation}
 where again $X_c$ is described by (\ref{X-eq-binary})-(\ref{X-eq7}). 
\end{theorem}

\begin{proof}
A formal proof is detailed in appendix. It essentially uses Propositions \ref{proposition-da} and \ref{proposition-dr} and the interpretation of the commitment price $\delta^*_c$ to show that the value of $u_c^*$ given by the market operator is optimal for the market participant $c$, and then to show that for this given value of $u_c^*$ fixed, the values of the $x_{ic}^*$ given by the market operator are also optimal for the market participant (relying on optimality conditions of the market operator welfare maximizing program and the market participant profit maximizing program where in both cases $u_c^*$ is fixed).
\end{proof}

Let us consider now the Example \ref{example1}.1 above and its optimal solution: the market price is set to 10 EUR/MW by the convex bid B which is fractionally accepted, and the commitment price associated to the constraint fixing the commitment binary variable $u_c=1$ is (- 330) EUR, corresponding to the incurred loss to unit C at the given commodity market price. With the instance \ref{example1}.2, the market price would be 40 EUR/MW and the commitment price set to (- 200) EUR, again corresponding to the incurred loss. These prices for the commodity and the commitments can readily be derived as the optimal dual variables $\pi$ and $\delta$ (in square brackets) in:

\textbf{Example 1 (continued)}

\begin{multicols}{2}
Example 1.1 (IP Pricing case):

\begin{center}

$\max_{x,u}\ (10)(300)x_a + (14)(10)x_b - (12)(40)x_c - (13)(100)x_d$
\begin{align*}
&\mbox{s.t.}\notag\\
&\quad 10x_a + 14x_b - 12x_c - 13x_d = 0 & [\pi^* = 10]\\
&\quad x_a \leq 1 \\
&\quad x_b \leq 1 \\
&\quad x_c \leq u_c \\
&\quad x_c \geq (11/12)u_c \\
&\quad x_d \leq 1 \\
&\quad u_c = 1 & [\delta^* = -330] \\
&\quad x\geq 0 \\
\end{align*}
\end{center}

\columnbreak

Example 1.2 (IP Pricing case):

\begin{center}
$\max_{x,u}\ (10)(300)x_a + (14)(10)x_b - (12)(40)x_c - (13)(100)x_d - 200 u_c$
\begin{align*}
&\mbox{s.t.}\notag\\
&\quad 10x_a + 14x_b - 12x_c - 13x_d = 0 & [\pi^* = 40] \\
&\quad x_a \leq 1 \\
&\quad x_b \leq 1 \\
&\quad x_c \leq u_c \\
&\quad x_d \leq 1 \\
&\quad u_c = 1 & [\delta^* = -200] \\
&\quad x \geq 0
\end{align*}
\end{center}

\end{multicols}

\vspace{-0.5cm}

Given these prices $(\pi, \delta) = (10, -330)$ or $(40, -200)$ respectively, participant C receives as a payment $\pi (-Q_c x_c) - \delta u_c$ (keeping the sign convention according to which $Q<0$ for sell orders), here respectively $10(11) - (-330)1 = 440$ or 40(10) - (-200)1 = 600. In each case, it corresponds to the production costs of C, and the primal decisions $(u_c, x_c)$ are optimal for the market participant C. They respectively solve the following profit-maximizing programs (cf. Theorem \ref{thm-ip-equilibrium}):
\vspace{-0.7cm}

\begin{multicols}{2}

\begin{center}
\begin{align}
&\max_{u_c, x_c} [\pi(12)x_c - \delta u_c] -  [ 40(12)x_c ] \label{profitmax-example1}\\
&\mbox{s.t.}\notag\\
&\quad x_c \leq u_c \\
&\quad x_c \geq (11/12)u_c \\
&\quad u_c \in \{0,1\} \label{profitmax-example1_end}\\
&\quad [\pi:=10, \delta := -330] \nonumber
\end{align}
\end{center}

\columnbreak

\begin{center}
\begin{align}
& \max_{u_c, x_c} [12\pi x_c -\delta u_c] - [12(40)x_c  + 200 u_c] \label{profitmax-example2}\\
&\mbox{s.t.}\notag\\
&\quad x_c \leq u_c \\
&\quad x_c \geq 0 \\
&\quad u_c \in \{0,1\} \label{profitmax-example2_end}\\
&\quad[\pi:=40, \delta := -200] \nonumber
\end{align}
\end{center}

\end{multicols}

It may obviously happen that the committed units (i.e. such that $u_c=1$) are profitable at the market price(s) $\pi$, in which case the optimal dual variable $\delta^*$ to the fixing constraint is positive. In such a case, if strictly applied, IP pricing would require a payment $\pi (-Q_c x_c) -\delta u_c$, where $-\delta u_c$ is negative and corresponds to a situation where the market participant gives its marginal rent back to the Market Operator and makes zero profits, similarly to a pay-as-bid scheme. However, as described in the original contribution \cite[p.282]{oneill} about the practice of the New-York Independent System Operator NYISO and Pennsylvania–New Jersey–
Maryland Interconnection (PJM), and also in \cite{Sioshansi2014}, market rules could specify that such profits can be kept by market participants. In such a setting, IP pricing could be described as ``marginal pricing plus make-whole payments" as only losses are compensated, while market participants can keep rents if any at the given electricity market prices. Note that this approach seems also close to the current practice in Ireland \cite[pp.40-43]{dicosmo2016}.

Let us emphasize that according to the payment scheme described in Definition \ref{def-settlement}, no payment is made to non-committed units since then $u=0$ and $x=0$. However, it may happen that rejected bids are profitable at the commodity market prices, in which case $\delta$ is positive. In that situation, the term $- \delta u$ in the settlement rule makes the market participant indifferent to being committed or not: if $u$ was switched to one to allow a profitable generation of electricity, a corresponding payment from the market participant to the market operator would occur offsetting these potential profits. This is another - maybe surprising - aspect of the underlying idea of Theorem 2 in \cite{oneill}, and the fact that for the obtained price system, optimal primal variables of the welfare program are also solving the market participant's profit-maximizing programs like (\ref{profitmax-example1})-(\ref{profitmax-example1_end}) or (\ref{profitmax-example2})-(\ref{profitmax-example2_end}).

On the other hand, concerning uplifts and considering the bidding products proposed in Europe (so-called block orders), the reference \cite{oneill2007} shows that with IP pricing, provided the welfare is positive, a welfare maximizing solution is always such that there is enough welfare to finance compensations paid to bids losing money, so-called ``paradoxically accepted block orders," if they are allowed. This is also discussed in \cite{madani2015mip} and can be seen here by observing that by strong duality, the welfare is equal to the dual objective in (\ref{swp-fixed-dual-obj}): if the welfare is positive, this means that the sum of the $\delta_{c_a}$ such that $\delta_{c_a} >0$ is greater than the sum of the negative $\delta_{c_a}$ which correspond to losses to compensated. Hence, there is always enough economic surpluses to compensate these losses.

Finally, one recurring critique of the IP pricing approach is that it exhibits important commodity price volatility \cite{ring, ruiz2012}. Intuitively, the reason is  that the units which are marginal and hence setting this price -- and can have substantially different greater or lower marginal costs -- can quickly change with an increase of load. We argue here that it also leads to counter-intuitive market prices, as Example \ref{example2} shows.

\textbf{Example 2 (continued) : IP Pricing case}

In the case of IP Pricing which seeks marginal pricing without any further restriction but considers side payments to cover losses of market participants, the price would be ``stuck" between 30 EUR/MW and 40 EUR/MW, respectively because of the acceptance of A and the rejection of C. However, we have seen that a price of 75 EUR/MW would make sense as it supports a market equilibrium once $C$ is removed from the instance, and would appear as ``fair" to all other market participants, as in the discussion in Section \ref{subsec-examples}.

This potentially counter-intuitive outcome is partially related to an arbitrary distinction between bids including non-convexities and those which don't, and the fact that convex bids cannot be paradoxically rejected, while non-convex bids can be. As a consequence, rejected convex bids impose conditions on market prices, while rejected non-convex bids do not. 

It is also more generally related to the possibility for rejected bids, convex or not, to impact market prices; a property related to Property 4 in \cite{schiro2015}, namely the possibility for offline generators to set the market price (see Section D therein).

\section{European-like market rules} \label{section-europeanrules}

European market rules are intimately related to IP pricing proposed in \cite{oneill}. They can be generally described as IP Pricing plus the constraints that all start up prices (or commitment prices) $\delta_{c_a}$ of \emph{committed} plants - or more generally accepted non-convex bids - must be positive or null. This means that a non-convex bid cannot be paradoxically accepted, while marginal bids are setting the price. More precisely, this means that instead of considering a partition of accepted and rejected non-convex bids (\ref{fix1cond})-(\ref{fix0cond}) corresponding to the welfare maximizing solution, one only considers those partitions for which the conditions just mentioned about the $\delta_{c}$ hold. Hence no make-whole payments are needed. Most of the time, the welfare maximizing partition is then not feasible. Also, non-convex bids can be paradoxically rejected and are not compensated for the corresponding opportunity costs. This corresponds to a situation where the optimal dual variable $\delta_c$ associated to the constraint of the form $u_c=0$ rejecting the bid is positive. Let us recall that according to the IP Pricing rule, rejected bids are not compensated, as the payment of the form $\pi (-Qx_c) - \delta_c u_c$ is null if $u_c=0$, see the exposition of IP pricing market rules above. The term $- \delta_c u_c$ in the objective just makes the participant indifferent to being committed or not at electricity market prices $\pi$ as there is no real opportunity costs according to the definition of the payment rule.

Under these European rules, in Examples \ref{example1}.1 and \ref{example1}.2, the bid $C$ must be rejected. Once rejected, the market price is increased to 100 EUR/MW, and the bid C is paradoxically rejected in both cases. The market outcomes in both cases is depicted on Figure \ref{fig-eu}.

\begin{figure}[ht]
\begin{center}
\includegraphics[scale=0.1]{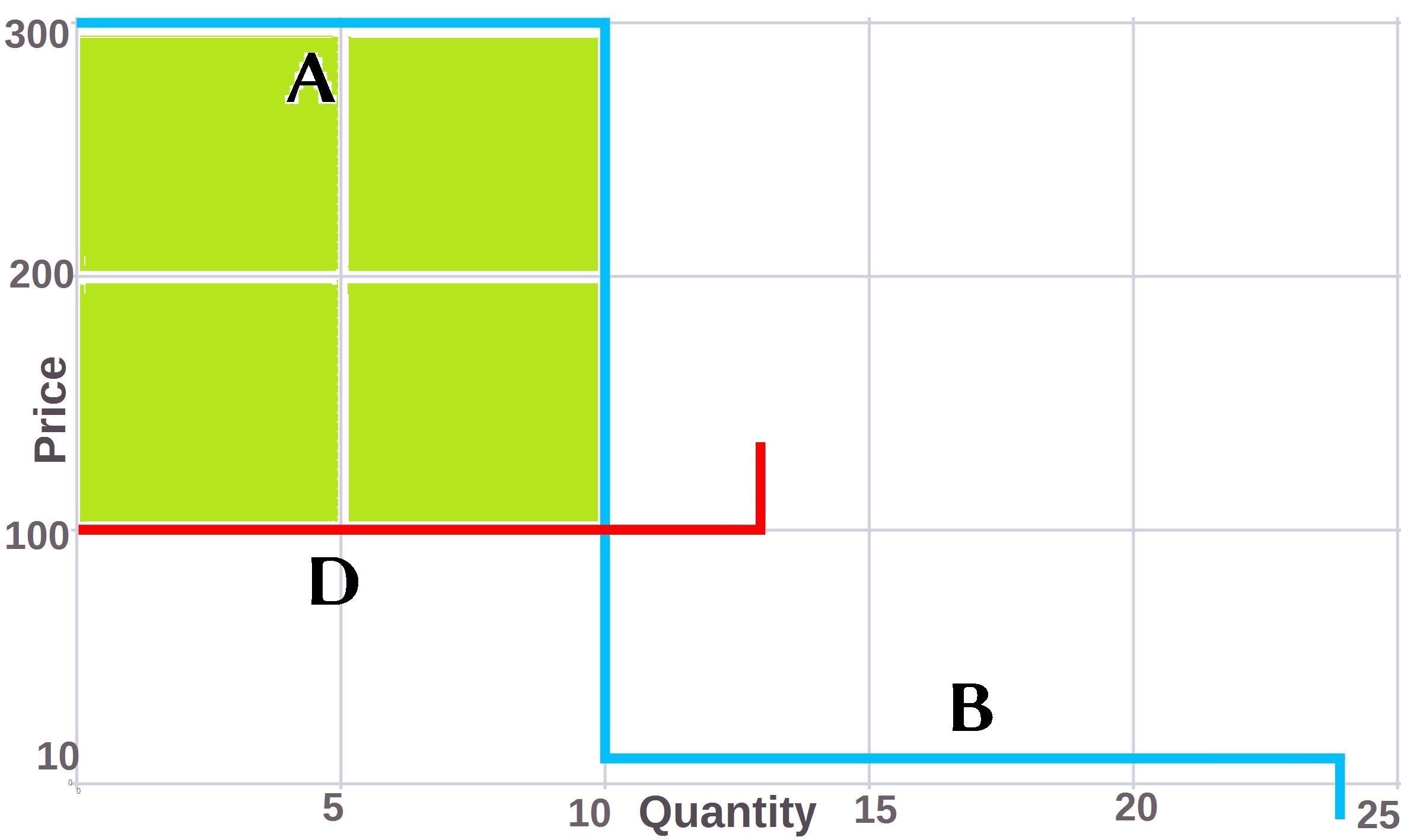}
\end{center}
\caption{A welfare sub-optimal solution satisfying European-like market rules for Example 1.1 and 1.2.}
\label{fig-eu}
\end{figure}

Let us note that in practice, there are hundreds of non-convex bids and only a limited fraction are paradoxically rejected. However, due to the increase of so-called block orders submitted in recent years, the number of these paradoxically rejected block orders has substantially increased and is a source of concern for all stakeholders. See \cite{madanivanvyve2016a} for a study using real Belgian market data from 2015.

Let us now consider Example \ref{example2}. The optimal solution under current European market rules is to fully accept A and B, and to reject C, D, E. The market price must then lie in the interval $[30;40]$ as A is fully accepted while C is a convex bid which is fully rejected and must hence be out-of-the money or at-the-money (i.e. the bid price must be  ``not good enough").

This can straightforwardly be shown using the following heuristic arguments. First, note that due to the bid quantities at hand, D is accepted if and only if E is accepted as well. However, if both are accepted, as no losses could be incurred, the market price must lie in the interval $[60;90]$. In this case, A, B and C are all strictly in-the-money and should be fully accepted, which leads to a contradiction as the balance constraint would be violated. So D and E must be both rejected which sets no particular condition on the outcome (as they are non-convex orders which can be paradoxically rejected), and it is then direct to check that given the bids A, B and C only, the market outcome is the one just described above. As for IP Pricing, such an outcome may be questionable.

These market rules lead to particularly interesting modeling and algorithmic issues related to peculiar MPEC and MILP models, which have been studied e.g. in \cite{mvv-revisiting, martin,madani2015mip, madani2014}.

\section{Numerical experiments} \label{section-numtests}

The numerical experiments are aimed at assessing the different microeconomic and computational tradeoffs offered by the historical pricing rules at hand. These experiments use realistic data comparable to instances coupling the Spanish and Portuguese day-ahead markets, with adaptations similar to those described in \cite{mvv-revisiting}: a minimum acceptance ratio has been added to the first step of the offer bid curves associated to a non-convex bid, and the marginal cost associated to this step takes into account information from an ad hoc variable cost which is provided in addition to the marginal cost curves in real Spanish and Portuguese instances. These instances contain bids with start-up costs, ramp constraints, step-wise marginal cost (resp. utility) curves and also a minimum power output level specified via the minimum acceptance ratio mentioned above. Though a two-node network is used here, more complex linear DC networks could be considered. The models and algorithms have been implemented in Julia (ver. 0.6.2) using the packages JuMP.jl (ver. 0.18.0) and CPLEX.jl (ver. 0.2.8), on a computer with an i7-8550u CPU (4 cores @ max 4 GHz) and 8 GB of RAM running Centos (RH Linux) 7, using CPLEX 12.7.1 as the underlying MIP solver. These models and algorithms, together with datasets, are provided in an online Git repository, in order to foster research on the topic \cite{dam_comp}.

\subsection{Welfare, side payments, and computational efficiency}

European-like market rules, which rely on uniform electricity prices, do not allow incurring losses to market participants at these given prices, and hence do not require uplifts, i.e., discriminatory payments, cf. Section \ref{section-europeanrules}. Avoiding such discriminatory payments has a cost in terms of total welfare, as illustrated by the examples in Section \ref{section-europeanrules}. It has also a cost in terms of computational efficiency: compared to IP Pricing, not all non-convex bids selections are admissible (see the discussion in Section \ref{section-europeanrules}) and the underlying combinatorial nature of the problem renders it challenging computationally, though efficient formulations and algorithms exist \cite{mvv-revisiting,euphemia,martin,madani2015mip,madani2014}. In the numerical tests below, for European-like market rules, the Benders decomposition described in \cite{mvv-revisiting} has been used.

IP Pricing requires us to solve the pure welfare maximizing problem, which as a MILP is comparatively much easier to solve for very large-scale instances, and then a basic LP. In the present context, the same holds for Convex Hull Pricing as the convex hull of market participants' feasible sets is given by the continuous relaxation. As a consequence, convex hull prices can be obtained by solving the (dual of the) continuous relaxation of the welfare maximization program which is an LP. However, in both cases, uplift payments are required to make participants whole and these discriminatory payments might be not well accepted by market participants or raise non-trivial implementation issues.

Tables \ref{table-num-1} and \ref{tab:runtimes} describe the characteristics of the instances used and presents numerical results illustrating the trade-offs between welfare, run times, and the amount of uplifts (side payments) required when IP Pricing or Convex Hull Pricing is used. The runtimes correspond in each case to solving the full problem, i.e., the time needed to determine both quantities and prices for the considered pricing rules. Let us note that the uplifts for Convex Hull pricing correspond to both actual losses and opportunity costs of market participants, while uplifts for IP Pricing only correspond to actual losses that should be compensated by definition of the settlement rule (see Section \ref{section-IPpricing}). It could hence be possible for an instance to have more uplifts reported for CHP than for IP Pricing, but counting the opportunity costs in the case IP Pricing is used would then show that CHP truly minimizes deviations from a market equilibrium with uniform prices.

As can be seen in Table \ref{table-num-1}, uplifts required are very small compared to the amount of welfare, and the same is true regarding the welfare losses when European pricing is used.

\begin{table}[htbp]
  \centering
  \caption{Welfares and uplifts (euros). The ``Welfare Loss (EU rules)" column indicates how much welfare is lost when European Pricing is used compared to the pure welfare optimal solution.}
    \begin{tabular}{ccccccc}
    Inst  & \# Non-Convex bids & \#Steps &Welfare & Welfare Loss & upliftsCHP & upliftsIP \\
      &   &  &(IP \& CHP) &  (EU rules)  &  &  \\
    1     & 90    & 14309 & 115426705.6 & 11084.8536 & 288.7258 & 7393.944 \\
    2     & 91    & 13986 & 107705738.5 & 5003.636 & 439.193 & 5000.8 \\
    3     & 91    & 14329 & 113999405.5 & 2141.15356 & 1030.314 & 6648.373 \\
    4     & 92    & 14594 & 109951139.7 & 9466.60112 & 603.5169 & 5827.93 \\
    5     & 89    & 14370 & 107172393.2 & 7754.3366 & 72.63568 & 867.284 \\
    6     & 87    & 14389 & 123823967.6 & 3377.139199 & 239.3088 & 1835.88 \\
    7     & 89    & 14783 & 119386085.4 & 6964.017 & 329.5143 & 3116.86 \\
    8     & 86    & 14414 & 105372099.8 & 2187.674081 & 72.25676 & 951.5828 \\
    9     & 88    & 14860 & 96023475.04 & 2046.41408 & 778.3553 & 5275.138 \\
    10    & 86    & 14677 & 98212635.81 & 2597.8314 & 401.637 & 2313.78 \\
    \end{tabular}%
  \label{table-num-1}%
\end{table}%

The instances at hand are ``easy instances" compared to instances which include thousands of so-called block bids, i.e., fully indivisible non-convex bids without associated start-up costs or ramp conditions and which are used in the NWE region (e.g., Great Britain, France, Belgium, Germany, Finland, Norway). For such more ``combinatorial" and challenging instances, the difference of run times to compute European Pricing outcomes versus IP Pricing or Convex Hull Pricing outcomes can be much more substantial, as numerical tests presented in \cite{madanivanvyve2016a} illustrate, while conclusions regarding welfare losses and required uplifts would essentially be the same: both welfare losses and required uplifts are rather small compared to the total welfare. The run times differences here are also due to aspects related to model generation: for IP Pricing, the continuous relaxation of the primal welfare maximizing program is used to compute prices (with the few additional constraints fixing binary variables to their value), while a full dual program must be built for European Pricing in order to test for the existence of electricity prices satisfying the constraints that no losses could be incurred (cf. the Benders decomposition described in \cite{mvv-revisiting}).

\begin{table}[htbp]
  \centering
  \caption{Run times for each pricing rule (in seconds)}
    \begin{tabular}{cccccc}
    Inst  & \# Non-convex bids & \# Steps & runEU & runIP & runCHP \\
    1     & 90    & 14309 & 4.047098 & 2.202199 & 2.073478 \\
    2     & 91    & 13986 & 4.648906 & 2.081456 & 2.065098 \\
    3     & 91    & 14329 & 4.231441 & 2.294439 & 2.102532 \\
    4     & 92    & 14594 & 4.82378 & 2.050598 & 2.345987 \\
    5     & 89    & 14370 & 4.410432 & 1.860187 & 1.819655 \\
    6     & 87    & 14389 & 3.78953 & 1.907919 & 2.25707 \\
    7     & 89    & 14783 & 4.631189 & 2.104128 & 2.149526 \\
    8     & 86    & 14414 & 3.8165 & 1.842994 & 2.142367 \\
    9     & 88    & 14860 & 4.603193 & 1.943571 & 2.043593 \\
    10    & 86    & 14677 & 3.73881 & 2.0862 & 1.897801 \\
    \end{tabular}%
  \label{tab:runtimes}%
\end{table}%

\subsection{Paradoxically accepted or rejected non-convex bids}

Paradoxically rejected bids correspond to a missed trading opportunity, that is the market participant could make more profits by selling (resp. buying) more. We hence take as a definition of \emph{paradoxically rejected non-convex bids}, bids such that in case of self-dispatch where the market participant optimizes its production (resp. consumption) taking into account only its own technical constraints and the electricity market prices, she or he could be more profitable than with the market operator's decisions, by selling (resp. buying) more power. This definition slightly reformulates the definition in terms of average prices usually given for ``paradoxically rejected block orders" in European markets, in order to make it applicable when Convex Hull Pricing is considered. When Convex Hull Pricing is used, a market participant can sometimes be more profitable by selling (resp. buying) less at some hours than what has been decided by the market operator, without being fully rejected. This situation cannot occur with European or IP Pricing as the continuous variable decisions made by the market operator are optimal for the market participant once its binary ``commitment variable" is considered fixed (roughly speaking, optimal commitment decisions are made by market operators, then marginal pricing is used).

The issue of paradoxically rejected bids (PRBs) in European markets is well-known and of concerns to market participants, see e.g. a discussion of the statistics of 2015 in \cite{madanivanvyve2016a}. Numerical results presented below tends to show empirically what intuition could suggest (see for example the toy examples discussed in Section \ref{section-IPpricing} and Section \ref{section-europeanrules}): their number is on average reduced in the case of IP Pricing and Convex Hull pricing, see Table \ref{tab:pabprb}. Experiments presented in \cite{madanivanvyve2016a} also show that this number is on average drastically reduced for very combinatorial instances involving a large number of fully indivisible orders (so-called block orders in European markets).

This reduction of PRBs is at the expense of having paradoxically accepted bids (PABs), i.e. bids incurring losses at the given electricity prices, which require uplifts to make participants whole. Though this raises implementation issues regarding the uplifts, this situation may be seen as preferable: a market participant is better with a paradoxically accepted bid for which an uplift is received than with an opportunity cost in case the bid is paradoxically rejected. Note that by definition, there are no paradoxically accepted bids with European Pricing. Let us also note that for the instance \#2, the fact that there is no PRB and no PAB with CHP doesn't mean there is a market equilibrium, and in fact Table \ref{table-num-1} shows that the total deviation from an equilibrium is 439.193. Such an outcome could correspond for example to a case where a market participant, while profitable, could be more profitable by selling \emph{less} power at some hours of the day while still satisfying its technical constraints: hence the corresponding bid is neither PAB nor PRB according to the definitions given above.

\begin{table}[htbp]
  \centering
  \caption{Number of paradoxically accepted (resp. rejected) non-convex bids for each pricing rule}
    \begin{tabular}{cccccccc}
    Inst  & \# Non-Convex bids & pabEU & prbEU & pabIP & prbIP & pabCHP & prbCHP \\
    1     & 90    & 0     & 2     & 1     & 0     & 1     & 1 \\
    2     & 91    & 0     & 1     & 1     & 0     & 0     & 0 \\
    3     & 91    & 0     & 5     & 1     & 0     & 0     & 1 \\
    4     & 92    & 0     & 2     & 1     & 0     & 1     & 5 \\
    5     & 89    & 0     & 4     & 1     & 0     & 0     & 0 \\
    6     & 87    & 0     & 1     & 2     & 0     & 1     & 1 \\
    7     & 89    & 0     & 2     & 1     & 0     & 1     & 1 \\
    8     & 86    & 0     & 2     & 1     & 0     & 0     & 2 \\
    9     & 88    & 0     & 2     & 2     & 0     & 0     & 3 \\
    10    & 86    & 0     & 2     & 1     & 0     & 0     & 1 \\
    \end{tabular}%
  \label{tab:pabprb}%
\end{table}%

\section{Conclusions}

Convex Hull Pricing, IP Pricing and European-like market rules have been compared in a power exchange setting with non-convex demand bids and bids including start-up costs, ramp constraints and minimum power output levels. For this purpose, it has been shown that Convex Hull prices can be  efficiently computed by considering the continuous relaxation of the welfare maximizing program and the associated dual variables when these stylized bidding products are considered. Convex Hull Pricing and IP Pricing have also been technically described with unified notation in such a two-sided auction setting, including the uplift minimizing property of Convex Hull prices and the equilibrium property of IP prices once the appropriate settlement rule is specified. This settlement rule relies both on electricity prices and on so-called start-up/commitment prices. European-like market rules should be considered as a variant of IP Pricing where only non-convex bid selections not incurring losses at the electricity market prices are considered, and this can be specified by imposing non-negativity constraints on so-called “start-up prices” associated with committed units.

The European approach has the advantage of avoiding so-called uplifts which are discriminatory payments that can raise implementation issues and concerns of market participants. Because there are no uplifts, there may be less market manipulation possible, though this is an intuitive statement. Investigating gaming opportunities under the different pricing rules is an interesting venue for further research. Such an issue is also highlighted in \cite{HUPPMANN2018622}, Section 5, or discussed in \cite{liberopoulos2016}. On the other side, the current European approach decreases welfare and is computationally much more challenging. Numerical experiments have been conducted on realistic instances to assess the trade-offs between the welfare losses in the case of European rules, and the uplifts required by either IP Pricing or Convex Hull Pricing. It turns out that both are relatively rather small compared to the total welfare. It also turns out that the number of so-called paradoxically rejected non-convex bids under European market rules is on average reduced when IP Pricing or Convex Hull Pricing is used. However, this doesn’t hold for all individual instances. This reduction of paradoxically rejected non-convex bids is at the expanse of allowing so-called paradoxically accepted bids which are bids which would incur losses given the electricity market prices, and hence require uplifts to make participants whole. However, as mentioned above, the total amount of uplifts required seem relatively small in practice. The outcomes IP and EU pricing exhibit on the toy examples introduced in Section \ref{subsec-examples} are also questionable. At the light of these toy examples and the numerical results, it seems that non-uniform pricing rules not relying on marginal pricing may be preferable. The reference \cite{dualpricing2017} proposes a so-called ``revenue neutral" pricing rule, a property which is also shared by the pricing rule introduced in \cite{vanvyve}. These options should be further investigated as they are computationally-efficient and could yield outcomes which are more relevant from an economic perspective.

Other pricing rules have been proposed recently in the literature, and a better assessment of the pros and cons of each approach in a real applied setting is still needed. As the pricing rules which have been considered here are historical pricing rules taken as key references in the literature, the present contribution aims at being a step in that direction.

\textbf{Acknowledgements:}

Dr. Mehdi Madani was supported by a Fellowship of the Belgian American Educational
Foundation.


\appendix

\section{Omitted proofs in main text}

\textbf{Proof of Lemma \ref{lemma-profit-loss} [adapted from \cite{mvv-revisiting}, Lemma 6 in Section 4]}
\begin{proof}
The complementarity conditions associated to (\ref{X-eq1})-(\ref{X-eq2}) are given by $s^{max}_{ic}(u_c - x_{ic})$ and $s^{min}_{ic}(x_{ic} - r_{ic}u_c)$ where $u_c=1$ as the bid considered is accepted, which directly gives that $s^{max}_{ic} x_{ic} = s^{max}_{ic}$ (*) and $s^{min}_{ic} x_{ic} = s^{min}_{ic} r_{ic}$ (**).

Multiplying  (\ref{cc-dual-xic}) by $x_{ic}$ and using the last two identities (*) and (**) yield:

\begin{equation}
s_{ic}^{max} - r_{ic} s_{ic}^{min} + (Q^{ic}g^{down}_{c,t(ic)-1} - Q^{ic}g^{up}_{c,t(ic)-1})x_{ic} + (Q^{ic}g^{up}_{c,t(ic)} - Q^{ic}g^{down}_{c,t(ic)})x_{ic} + Q_{ic} \pi_{t(ic)}x_{ic} = P^{ic}Q_{ic}x_{ic} \label{lemma-int-eq1}
\end{equation}

On the other hand, complementarity conditions associated to (\ref{X-eq6})-(\ref{X-eq7}) are respectively given by:

\begin{multline}
g^{up}_{c,t} \left( RU_c\ u_c - \sum_{ic \in I_c | t(ic) = t} (Q^{ic})x_{ic} + \sum_{ic \in I_c  | t(ic) = t+1 } (Q^{ic})x_{ic} \right) = 0 \\ \forall t \in \{1,...,T-1\}, \forall c \in C \hspace{0.7cm}  \label{X-eq6-cc}
\end{multline}

\begin{multline}
g^{down}_{c,t} \left(RD_c\ u_c - \sum_{ic \in I_c | t(ic) = t+1} (Q^{ic})x_{ic}  +\sum_{ic \in I_c | t(ic) = t} (Q^{ic})x_{ic} \right) = 0 \\ \forall t \in \{1,...,T-1\}, \forall c \in C \hspace{0.7cm}   \label{X-eq7-cc}
\end{multline}

Summing (\ref{X-eq6-cc}) to (\ref{X-eq7-cc}), then summing up over $t$ and using the fact that $u_c = 1$ as the bid considered is accepted provides after rearrangement:

\begin{equation}
\sum_{ic} \left( (Q^{ic}g^{down}_{c,t(ic)-1} - Q^{ic}g^{up}_{c,t(ic)-1})x_{ic} + (Q^{ic}g^{up}_{c,t(ic)} - Q^{ic}g^{down}_{c,t(ic)})x_{ic} \right) = \sum_t (RU_c g^{up}_{c,t} + RD_c g^{down}_{c,t}) \label{lemma-int-eq2}
\end{equation}

The result of the Lemma is then obtained by summing up (\ref{lemma-int-eq1}) over the $ic \in I_c$ and using the identity (\ref{lemma-int-eq2}). \end{proof}

\textbf{Proof of Theorem \ref{thm-ip-equilibrium}}

\begin{proof}
Part I: $u_c^*$ is optimal.

In the first case where $u_c^* = 1$, the alternative $u_c^-= 0$ yields a zero profit while for the current value $u_c^* = 1$, as $\delta^*_c = \sum_{ic_a} -Q_{ic_a}(\pi_{t(ic_a)}-P^{ic_a})x^*_{ic_a} - F_{c_a}$ according to Proposition \ref{proposition-da}, the profit maximizing objective in (\ref{indiv-obj}) evaluated at $(u^*_c, x^*_{ic})$ is also 0 and hence $u_c^* = 1$ is at least a value as good as $u_c^- = 0$ (Part II below will show that it is indeed not possible to improve the profit value with $u_c^* = 1$ as the $x_{ic}^*$ are also optimal). Note that when $\delta^*_c >0$, it corresponds to a market participant profitable at the electricity market prices $\pi$ and the part of the payment given by $-\delta^*_c u_c^*$ corresponds to a marginal rent given back to the market operator: market rules can obviously specify that this marginal rent can be kept by market participants.

In the second case where $u_c^* = 0$ and the profit is zero, considering the alternative $u_c^- = 1$ and the associated optimal decisions $x^-_{ic}$, according to Proposition \ref{proposition-dr}, $\delta^*_c \geq \sum_{ic} -Q_{ic}(\pi_{t(ic)}-P^{ic})x^-_{ic} - F_c$ and hence the profit maximizing objective in (\ref{indiv-obj}) evaluated at $(u^-_c, x^-_{ic})$ is non positive: $u_c^* = 0$ is an optimal decision.

Part II: the $x_{ic}^*$ are optimal for $u_c^*$ fixed.

Let us first note that in the case where $u_c^* = 0$, the $x_{ic}$ are in all cases constrained to be null and there is nothing to show (the values $x^*_{ic}=0$ are trivially optimal).

In all cases, let us consider (\ref{indiv-obj})-(\ref{indiv-conds}) where (\ref{indiv-conds}) is given by the conditions (\ref{X-eq1})-(\ref{X-eq2}), (\ref{X-eq6})-(\ref{X-eq7}) and where $u_c$ is considered as a \emph{parameter} with the given value $u_c^*$: hence, the objective (\ref{indiv-obj}) to maximize is reduced to $\sum_{ic\in I_c}   [P^{ic} - \pi^*_{t(ic)}]  Q^{}_{ic}x_{ic}$. Now let us consider the optimality conditions for this problem given by the primal, dual and complementarity conditions. These dual conditions are:

\begin{align}
&s_{ic}^{max} - s_{ic}^{min} + (Q^{ic}g^{down}_{c,t(ic)-1} - Q^{ic}g^{up}_{c,t(ic)-1}) + (Q^{ic}g^{up}_{c,t(ic)} - Q^{ic}g^{down}_{c,t(ic)}) = (P^{ic} - \pi_{t(ic)})Q_{ic} \\
&s^{max}, s^{min} \geq 0
\end{align}

And the complementarity conditions are only those associated to the primal conditions (\ref{X-eq1})-(\ref{X-eq2}), (\ref{X-eq6})-(\ref{X-eq7}), for the given $c$.

It is then straightforward to check that all these optimality conditions are part of the optimality conditions for (SWP-FIXED) given by the objective (\ref{lin-primal-obj}) and conditions (\ref{chp-primal-2}), (\ref{X-eq1})-(\ref{X-eq2}), (\ref{X-eq6})-(\ref{X-eq7}) and (\ref{fix1cond})-(\ref{fix0cond}).

Hence, the values $x^*_{ic}$ provided by the market operator are optimal for the profit maximizing program of participant $c$ with $u_c$ fixed. 

Part I and Part II together proves Theorem \ref{thm-ip-equilibrium}.
\end{proof}

\bf{References}

\bibliography{template}

\bibliographystyle{plainnat}

\end{document}